\documentclass[sn-mathphys-num]{sn-jnl}% Math and Physical Sciences Numbered Reference Style 
\usepackage{graphicx}%
\usepackage{multirow}%
\usepackage{amsmath,amssymb,amsfonts}%
\usepackage{amsthm}%
\usepackage{mathrsfs}%
\usepackage[title]{appendix}%
\usepackage{xcolor}%
\usepackage{textcomp}%
\usepackage{manyfoot}%
\usepackage{booktabs}%
\usepackage{algorithm}%
\usepackage{algorithmicx}%
\usepackage{algpseudocode}%
\usepackage{enumitem}
\usepackage{listings}%
\newcommand{\E}{\mathcal{J}}
\newcommand{\R}{\mathbb R}

\newcommand{\N}{\mathbb N}

\def\R{\mathbb{R}}

\def\> {\rightarrow}
\def\S1{\mathbb S^1}

\renewcommand{\div}{\mathrm{div }}

\newcommand{\divs}[1]{{\rm div}^s\left(#1\right)}
\newcommand{\inn}[2]{\left\langle #1,#2\right\rangle}
\newcommand{\poincConst}{\lambda^s_1(\Omega)}
\newcommand{\gam}[1]{\gamma\left(#1\right)}
\newcommand{\Gam}[1]{\Gamma\left(#1\right)}
\newcommand{\gmin}{\gamma_{\rm min}}
\newcommand{\gmax}{\gamma_{\rm max}}
\newcommand{\ginf}{\gamma(\infty)}

\theoremstyle{thmstyleone}%
\newtheorem{theorem}{Theorem}%  meant for continuous numbers
\newtheorem{proposition}[theorem]{Proposition}% 

\theoremstyle{thmstyletwo}%

\theoremstyle{thmstylethree}%
\newtheorem{lemma}{Lemma}%

\begin{document}

\title[Existence of solutions to a quasilinear nonlocal PDE]{Existence of solutions to a quasilinear nonlocal PDE}


\author*[1]{\fnm{Lisbeth} \sur{Carrero}}\email{lisbeth.carrero@uoh.cl}
\equalcont{The three authors contributed equally to this work.}

\author[2]{\fnm{Alexander} \sur{Quaas}}\email{alexander.quaas@usm.cl}
\equalcont{The three authors contributed equally to this work.}

\author[1]{\fnm{Andres} \sur{Zuniga}}\email{andres.zuniga@uoh.cl}
\equalcont{The three authors contributed equally to this work.}

\affil*[1]{\orgdiv{Instituto de Ciencias de la Ingenier\'ia}, \orgname{Universidad de O'Higgins}, \orgaddress{\street{Avenida Alameda 611}, \city{Rancagua}, \country{Chile}}}

\affil[2]{\orgdiv{Departmento de Matem\'atica}, \orgname{Universidad T\'ecnica Federico Santa Mar\'ia}, \orgaddress{\street{Avenida Espana 1680}, \city{Valpara\'iso}, \country{Chile}}}


\abstract{In this paper, we introduce a new class of quasilinear operators, which represents a nonlocal version of the operator studied by Stuart and Zhou~\cite{stuart}, inspired by models in nonlinear optics. We will study the existence of at least one or two solutions in the cone $X := \{ u \in H^{s}_0(\Omega) : u \geq 0 \}$ using variational methods. For this purpose, we analyze two scenarios: the asymptotic sublinear and linear growth cases for the reaction term. Additionally, in the sublinear case, we establish a nonexistence result.}




\maketitle
%\tableofcontents

\section{Introduction}\label{sec1}
Before introducing the main problem, we will review some classic and recent concepts regarding the theory of fractional differential operators.

\subsection{A brief review of nonlocal differential operators.}
Fractional differential operators have emerged as fundamental tools in modeling complex phenomena that exhibit nonlocal properties. These operators are particularly effective for capturing global effects and long-range interactions, making them indispensable in diverse fields such as continuum mechanics, finance, biology, physics, population dynamics, game theory, machine learning, and advanced approaches to hyperelasticity, among others (see \cite{Laskin, Ralf, fisica2, fisica, antil, bellido}).
Among these operators, the fractional Laplacian has been the subject of extensive study in the literature since early 2000's. This operator can be defined for $s\in (0,1)$ via its multiplier $|\xi|^{2s}$ in the Fourier space. For suitable functions like $u\in L^{\infty}(\mathbb{R}^{d})\cap C^{2s+ \beta}_{loc}(\mathbb{R}^{d})$ it can also be defined through the singular integral
\begin{equation}\label{fracLapl}
    (-\Delta)^s u(x) := C_{d, s} \;\mathrm{P.V.} \int_{\mathbb{R}^d} \frac{u(x) - u(z)}{|x - z|^{d + 2s}}\,dz,
\end{equation}
for a well-known normalizing constant $C_{d,s}>0$, and where P.V. stands for the principal value in the singular integral. As a matter of fact, there are at least ten equivalent definitions of this fractional Laplacian operator, including the so-called semi-group definition, the celebrated harmonic extension definition by Silvestre and Caffarelli~\cite{caffarelli2007}, the Bochner definiton, Dynkin's definition, among others; see \cite{kwasnicki2017} for a detailed survey on this matter. 

The fractional Laplacian together with the integration by parts formula gives rise to natural notion of fractional Sobolev space $W^{s,p}(\Omega)$, for $p\in[1,+\infty)$ and an open subset $\Omega$ of $\R^d$. It is traditionally defined as an intermediary Banach space between $L^p(\Omega)$ and $W^{1,p}(\Omega)$ via the Gagliardo seminorm
\[
    W^{s,p}(\Omega):=\left\{u\in L^p(\Omega):\; \|u\|_{W^{s,p}(\Omega)}<\infty\right\};
\]
endowed with the natural norm $\|u\|_{W^{s,p}(\Omega)}:=\left(\|u\|^p_{L^p(\Omega)}+[u]^p_{W^{s,p}(\Omega)}\right)^{\frac 1p}$, where
\[
    [u]_{W^{s,p}(\Omega)}:=\left(\int_{\Omega}\int_{\Omega}\frac{|u(x)-u(y)|^p}{|x-y|^{d+sp}}\,dxdy\right)^{\frac{1}{p}}.
\]
A very gentle, yet complete, introduction to the subject of fractional Sobolev spaces can be found in~\cite{valdinoci2011}. The interested reader is also referred to a couple of recently published books on the subject of fractional Sobolev spaces:~\cite{fractional2023} for the connection with functional inequalities, and~\cite{first2023} where the classical Sobolev theory is adapted to this fractional context with applications to wavelets and existence and regularity theory for the Poisson problem.

On the other hand, the notions of \emph{fractional gradient and fractional divergence operator} are relatively new concepts that are still in their early stages of development. The fractional $s$-gradient of a function \(u \in C_{c}^{\infty}(\mathbb{R}^{d})\), for \(s \in (0,1)\), is defined as
\begin{equation}\label{eq:intro:fract_grad}
  \nabla^{s}u(x):= \mu_{n,s} \int\limits_{\mathbb{R}^{n}} \dfrac{(y-x)(u(y)-u(x))}{|y-x|^{d+s+1}}\,dy,
\end{equation}
where $\mu_{d,s}$ is a multiplicative normalization constant controlling the behavior of $ \nabla^{s}u$ as $s\to 1^{-}$. An important feature of the fractional gradient~\eqref{eq:intro:fract_grad} is that it satisfies three natural qualitative requirements as a fractional operator: invariance under translations and rotations, homogeneity of order $s$ under dilations and some continuity properties relative to the Schwartz's space $\mathcal{D}(\R^d)$. A fundamental result shown by ${\rm\check{S}}$ilhav\'y~\cite{sil} is that, in fact, these requirements characterize the fractional gradient, up to a multiplicative constant; see~\cite[Theorem 2.2]{sil}. Whence, from both \emph{a physical and mathematical point of view}, the definition~\eqref{eq:intro:fract_grad} is well-posed. A detailed account on the existing literature on this operator can be found in~\cite[Section 1]{shieh2018}.

Following a similar line as in~\cite{shieh2015}, we define the associated fractional Sobolev space via approximation, as the closure
\[
    X_{0}^{s,p}(\mathbb{R}^{d}):= \overline{C_{c}^{\infty}(\mathbb{R}^{d})}^{\|\cdot\|_{X^{s,p}(\mathbb{R}^{d})}}
\]
endowed with the norm
\[
    \|u\|_{X^{s,p}(\mathbb{R}^{d})} := \left(\|u\|^p_{L^{p}(\mathbb{R}^{d})} + \|\nabla^{s} u\|^p_{L^{p}(\mathbb{R}^{d};\mathbb{R}^{d})}\right)^{\frac 1p},
\]
for all $u\in C_{c}^{\infty}(\mathbb{R}^{d})$. By~\cite[Theorem 2.2]{shieh2015}, we know that
\[
    X^{s+\epsilon,p}_0\subset W^{s,p}(\R^d)\subset X^{s-\epsilon, p}_0(\R^d)
\]
with continuous embeddings for all $s\in (0,1)$, $p\in(1,\infty)$ and $0<\epsilon<\min\{s,1-s\}$. In particular,~\cite[Theorem 2.2]{shieh2015} asserts that in the Hilbert case $p=2$ there actually holds, for any $s\in (0,1)$,
\begin{equation}\label{eq:intro:aux:equivalence}  
    X_{0}^{s,2}(\mathbb{R}^{d})\simeq H^{s}(\mathbb{R}^{d}),\;\text{ with }\;\; H^s(\Omega):=W^{s,2}(\R^d);
\end{equation}
In light of the equivalence~\eqref{eq:intro:aux:equivalence}, there exist positive constants \(C_1\) and \(C_2\) such that
\begin{equation*}
  C_1 \|u\|_{H^{s}(\mathbb{R}^{d})}\leq \|u\|_{X^{s,2}(\mathbb{R}^{d})} \leq C_2\|u\|_{H^{s}(\mathbb{R}^{d})},\quad\forall u\in X^{s,2}_0(\R^d).
\end{equation*}

Furthermore, a similar characterization holds for the so-called \emph{fractional divergence operator} for vector-fields. Namely, for any $s\in (0,1)$,
\[
    \text{div}^{s}\varphi(x):= \mu_{n,s} \int\limits_{\mathbb{R}^{d}} \dfrac{(y-x)\cdot (\varphi(y)-\varphi(x))}{|y-x|^{d+s+1}}dy,\quad\forall \varphi \in \mathcal{D}(\R^d;\R^d).
\]
As one could expect, the nonlocal gradient and nonlocal divergence are naturally related to the classical fractional Laplacian (see~\cite[Theorem 5.3]{sil}), since in particular,
\[
    -\text{div}^{s}\nabla^{s}u(x)= (-\Delta)^s u(x). 
\]
Moreover, as it is observed in \cite[Section 6]{shieh2015}, the fractional and fractional divergence \emph{are dual}, in the sense that
\[
    \int_{\R^d}u\, \div^{s}\varphi\,dx=-\int_{\R^d}\inn{\varphi}{\nabla^su}\,dx,   
\]
for all $u\in C^{\infty}_{c}(\R^d)$ and $\varphi\in C^{\infty}_{c}(\R^d;\R^dm,4)$, where here and henceforth $\inn{\cdot}{\cdot}$ denotes the Euclidean inner product in $\R^d$. A detailed account of these new fractional differential operators can be found in the seminal work of Shieh and Spector~\cite{shieh2015,shieh2018}, Comi and Stefani~\cite{gio}, and ${\rm\check{S}}$ilhav\'y~\cite{sil}.\smallskip

A different point of view that has been used, over the past decade, to define similar operators (which turn out to be equivalent to those aforementioned). It is based on exploiting the continuum nature of the fractional exponent $s\in (0,1)$ over the fractional Laplacian operator. Concretely, given $s\in (0,1)$ we can define a couple of \emph{signed fractional gradient operators} as follows 
\[
    \nabla^s_+:=(-\Delta)^{\frac s2} \quad\text{ and }\quad \nabla^s_-:=\nabla(-\Delta)^{\frac{s-1}2}. 
\]
If $u\in C^{0,s+\epsilon}_{loc}(\R^d)$ for some $\epsilon\in (0,1-s)$ and $(1+|x|)^{d+s}u(x)\in L^{1}(\R^d)$, then
\[
    \nabla^s_+u(x)=c_{n,s,+}\int_{\R^d}\frac{u(x)-u(y)}{|x-y|^{d+s}}dy \quad\text{ and }\quad \nabla^s_-u(x)=c_{n,s,-}\int_{\R^d}\frac{(x-y)(u(x)-u(y))}{|x-y|^{d+s+1}}dy 
\]
The fractional couple $\nabla^s_{\pm}$ is related to the dual pair of fractional divergence couple, via
\[
    +\div^s_+:=[\nabla^s_+]^*=(-\Delta)^{\frac s2}\quad\text{ and }\quad -\div^s_-:=[\nabla^s_-]^*=\div(-\Delta)^{\frac{s-1}2}.
\]
Foundational theory on these operators and related ones, including calculus rules for $\nabla^s_{\pm}$ and $\pm\div^s_{\pm}$, has been developed by different groups of mathematicians; we refer the interested reader to~\cite{biler2015, caffarelli2011,capella2011,liu2021-2,liu2022, mazowiecka2018, shieh2015, shieh2018, sil}, and the references therein. 

On a very recent article~\cite{liu2024}, Liu, Sun and Xiao prove that existence of weak solutions for the two versions of the nonlocal $p$-Laplacian operator, namely,
\[
    \Delta_{s,(p,\pm)}u:=\pm\div^s_{\pm}(|\nabla^s_{\pm}u|^{p-2}\nabla^s_{\pm}u)
\]
is indeed equivalent to either the analytic Sobolev inequality of fractional order, or the geometric isocapacitary inequality of fractional order. Furthermore, the authors construct an explicit fundamental solution, in the entire space, to both fractional $(p,+)$-Laplacian and $(p,-)$-Laplacian, as well as to establish Liouville-type principles for these fractional operators.

The topic of general nonlocal PDEs (quasilinear and fully nonlinear type) has slowly grown over the years to a hot topic, with plenty of open problems due to the mathematical difficulties associated with the study of nonlocal objects.

\subsection{Motivation of the problem and main results.}
In this article, we focus on a new class of nonlocal operators, drawing inspiration from Stuart's and Zhou's contributions. More specifically, we are interested in studying an appropriate nonlocal version of the second order quasilinear operator defined by
\begin{equation*}
  \emph{L}u := - \text{div} \left[ \gamma \left( \frac{u^2 + |\nabla u|^2}{2} \right) \nabla u \right] + \gamma \left( \frac{u^2 + |\nabla u|^2}{2} \right) u,
\end{equation*}
where $\gamma: [0, \infty) \to \mathbb{R}$ is a positive continuous function. In the early 2000's, Stuart and Zhou introduced this operator (see~\cite{shsz, zhou, cas}) in the study of the eigenvalue problem with boundary condition
\begin{equation}\label{stu}
  \left\{
  \begin{array}{rll}
   \emph{L}u &= \lambda u + h  & \;\; \text{in} \;\;\Omega, \\
   u &= 0 & \;\; \text{on} \;\;\partial\Omega.
  \end{array}
\right.\\[.25em]
\end{equation}
The physical motivation for stuying such a problem comes from the mathematical modeling in nonlinear optics, where the boundary value problem~\eqref{stu} describes the propagation of a self-trapped beam in a cylindrical optical fiber made of a self-focusing dielectric material. This PDE is obtained by reducing Maxwell's systems of equations in electroestatics to a single equation, under an ansatz of radial symmetry over planes perpendicular to the optical fiber (for more details, refer to~\cite{shsz}).

Building on Stuart-Zhou's pioneering work, Jeanjean and Radulescu extended the investigation to more general settings. In~\cite{radulescu}, they considered equations of the type
\begin{equation}\label{jr}
  \left\{
  \begin{array}{rll}
   \emph{L}u &= f(u) + h  & \;\; \text{in} \;\;\Omega, \\
   u &= 0 & \;\; \text{on} \;\;\partial\Omega.
  \end{array}
\right.
\end{equation}
where $f$ is a given continuous function and $h \in L^2(\Omega)$ is non-negative. The task of searching for solutions to these equations required the use of more sophisticated variational techniques, reflecting the challenges and insights first presented by Stuart and Zhou. 

This type of problem has attracted the interest of the PDE community. Some recent state-of-the-art results involving generalizations of the operator $L$, include~\cite{pomponio, wantin}.\smallskip

We seek to extend the results of Jeanjean and Radulescu to the nonlocal fractional setting. As a starting point, we analyze the problem in a simplified form, where the operator depends only on $\nabla^{s} u$. Our primary objective is to investigate the existence of solutions, in dimensions $d \geq 2$, to the following quasilinear nonlocal partial differential equation:
\begin{equation}\label{prob}
\left\{
\begin{array}{lll}
    -\!&\divs{\gam{\dfrac{|\nabla^su|^2}{2}}\nabla^su}=f(u)+h
    & \text{ in }\;\Omega \\
     & u=0 & \text{ in }\;\R^d\setminus\Omega
\end{array}
    \right.    
\end{equation}
where \(\gamma: [0, \infty) \to \R\) is a nonlinear heterogeneity in the fractional diffusivity with exponent \(s \in (0,1) \), \(h \in L^2(\Omega)\) is non-negative right side, and $f:\R\to\R$ is a continuous reaction term that exhibits (some of) the following behavior:\smallskip
\begin{enumerate}[label= $(\pmb{f\arabic*})$,itemsep=0.25em]
    \item\label{item:f:1} $\limsup\limits_{t\to 0}\dfrac{f(t)}{t}<\gmin\,\poincConst$,
    \item\label{item:f:2}$\limsup\limits_{t\to+\infty}\dfrac{f(t)}{t^p}=0$, for some $p\in(1,2^*_s-1)$.
    \item\label{item:f:3} $\liminf\limits_{t\to+\infty}\dfrac{f(t)}{t}\geq\ginf\, \poincConst$,
    \item \label{item:f:4}$\liminf\limits_{t\to+\infty}\dfrac{f(t)}{t}<\infty$.\\
\end{enumerate}
where $\poincConst$ denotes the first eigenvalue of fractional laplacian $(-\Delta)^s$ over $\Omega$.

Concerning the coefficient $\gamma$, it satisfies 
\begin{enumerate}[label=$(\pmb{\gamma\arabic*})$,itemsep=0.25em]
  \item\label{item:gamma:1} $\gamma\in C([0,\infty))$ and there are positive constants $\gmin,\gmax$ such that
    $$
        0<\gmin \leq \gamma(t) \leq\gmax,\qquad \forall t \in [0, + \infty).
    $$
    In addition, the limit $\ginf:=\lim\limits_{t\to+\infty}\gamma(t)\in(0,\infty)$ exists.
  \item\label{item:gamma:2} The map $t \mapsto \Gamma(t^{2})$ is convex in $[0, + \infty)$, where $\Gamma(t):= \int\limits_{0}^{t} \gamma(s)ds$.
\end{enumerate}
An example of such coefficient $\gamma$ is given by
 $$
 \Gamma(t)= At +B[(1+t)^{p/2}-1],
 $$
 where $A$ y $B$ are positive constants, and $1<p<2$.
 
 Let us now observe that, in view of~\ref{item:f:4}, there exists \( t_0 > 0 \) in such a way that
\[
    f(t) \leq C t, \;\;\text{ for }\; t> t_0.
\] 
Our variational line of arguments require us to assume that there exists \( R \geq t_0 \) for which the following condition holds
\begin{equation}\label{condbound}
R^2\geq CR^2 + \|h\|_{L^2(\Omega)}R.
\end{equation}

Our first main result states the following:\smallskip

\begin{theorem}\label{thm:1}
Assume that $h\in L^2(\Omega)$, $h\geq 0$  and that the condition \eqref{condbound} holds. Then,\smallskip 
    \begin{enumerate}[label=\roman*.,itemsep=0em]
    \item If conditions~\ref{item:f:1}-\ref{item:f:2} and~\ref{item:gamma:1} are satisfied, then for $ \|h\|_{L^2(\Omega)}$ sufficiently small, Problem~\eqref{prob} admits at least one non-negative solution. Furthermore, this solution is nontrivial if $h\ne 0$.
        \item If conditions~\ref{item:f:1}-\ref{item:f:3}-\ref{item:f:4} and~\ref{item:gamma:2} are satisfied, then for $ \|h\|_{L^2(\Omega)}$ sufficiently small, the non-homogeneous problem \eqref{prob}$(\text{i.e. }h\ne 0)$ admits at least two non-negative solutions, and the homogeneous problem~\eqref{prob} $(\text{i.e. }h=0)$ admits at least one nontrivial non-negative solution.
    \end{enumerate}
\end{theorem}
\medskip

The techniques we use to prove the Theorem~\ref{thm:1} are variational in nature. We will establish the existence of two solutions: one by minimizing the energy functional and a second one by invoking the mountain pass theorem. An interesting observation is that we cannot directly apply the usual techniques in the space $H_0^s(\Omega)$, due to a \emph{lack of a maximum principle} for the associated nonlocal quasilinear operator. As far as the authors are aware, this issue remains an open problem. Moreover, another key observation consists of noting that the classical variational strategy for establishing the nonnegativity of solutions cannot be replicated in this fractional framework. In order to prove that a solution \( u \) of \eqref{prob} is nonnegative, it requires that \(\int_{\R^d} \inn{\nabla^{s} u_{+}(x)}{\nabla^{s} u_{-}(x) }dx \neq 0\).

Indeed, in the one-dimensional case, let us consider the identity function. Notice that\smallskip
\begin{equation*}
\nabla^{s} u_{+}(x) = \int_{-\infty}^{0} \dfrac{-(y-x)x}{|y-x|^{n+s+1}}dy + \int_{0}^{+\infty} \dfrac{(y-x)^2}{|y-x|^{n+s+1}}dy>0, \quad \textbf{if} \quad x>0.
\end{equation*}
\begin{equation*}
\nabla^{s} u_{+}(x) = \int_{0}^{+\infty} \dfrac{(y-x)y}{|y-x|^{n+s+1}}dy>0, \quad \textbf{if} \quad x<0.
\end{equation*}
\begin{equation*}
\nabla^{s} u_{-}(x) = \int_{-\infty}^{0} \dfrac{-(y-x)x}{|y-x|^{n+s+1}}dy<0, \quad \textbf{if} \quad x>0.
\end{equation*}
\begin{equation*}
\nabla^{s} u_{-}(x) = \int_{-\infty}^{0} \dfrac{-(y-x)^2}{|y-x|^{n+s+1}}dy + \int_{0}^{+\infty} \dfrac{(y-x)x}{|y-x|^{n+s+1}}dy<0, \quad \textbf{if} \quad x<0.
\end{equation*}

Hence, 
\begin{equation*}
\int_{\R} \nabla^{s} u_{+}(x)\nabla^{s} u_{-}(x)= \int_{0}^{+\infty} \nabla^{s} u_{+}(x)\nabla^{s} u_{-}(x) + \int_{-\infty}^{0} \nabla^{s} u_{+}(x)\nabla^{s} u_{-}(x)<0.
\end{equation*}

Since we are interested in positive solutions, we will seek a solution in the set 
\[
    X := \{ u \in H^{s}_0(\Omega) : u \geq 0 \}.
\]
We will apply some results by Precup and Varga~\cite{radu} to guarantee the existence of solutions in \( X \).\\

In addition, we study the case when $f$ has asymptotic sublinear growth, that is, $f(t)=\nu g(t)$ for some $\nu >0$, where $g: \R \to \R$ is a continuous function satisfying \smallskip

\begin{enumerate}[label= $(\pmb{g\arabic*})$,itemsep=0.25em]
    \item\label{item:g:1} $g(t)=o(t)$ as $t\to + \infty$.
    \item\label{item:g:2} $G(t_0):= \int^{t_0}_{0} g(r)dr >0$, for some $t_0>0$.
    \item\label{item:g:3} $|g(t)|=O(|t|)$.    
\end{enumerate}
\smallskip

In the case, by a similar argument to that used to obtain \eqref{condbound}. There exists \( t_0 \), obtained from \ref{item:g:1}, such that there is an \( R \geq t_0 \) for which the following holds

\begin{equation}\label{condboundg}
R^2 \geq C R^2 + \|h\|_{L^2(\Omega)} R.
\end{equation}
Our second main result is the following:\smallskip

\begin{theorem}\label{thm:2}
Assume that conditions ~\ref{item:gamma:1}-\ref{item:gamma:2} are satisfied, $h\in L^2(\Omega)$, $h\geq 0$ with $\|h\|_{L^2(\Omega)}$ sufficiently small and that the condition \eqref{condboundg} holds. Then,\smallskip
 \begin{enumerate}[label=\roman*.,itemsep=0.25em]
    \item If~\ref{item:g:1} holds and $h\ne 0$. Then, Problem~\eqref{prob} has at least one non-negative solution, for every $\nu >0$.
        \item  In the case $h\equiv 0$,\smallskip  
          \begin{enumerate}[label={\em(\alph*)},itemsep=0.25em]
              \item If conditions ~\ref{item:g:1}-\ref{item:g:2} are satisfied, there exists $\nu_1 >0$ so that~\eqref{prob} has at least one non-negative nontrivial solution, for any $\nu>\nu_1$.
              \item If conditions ~\ref{item:g:1}-\ref{item:g:3} are satisfied, there exists $\nu_0>0$ so that~\eqref{prob} does not admit nontrivial solutions, provided $0<\nu <\nu_0$.
          \end{enumerate}
    \end{enumerate}
\end{theorem}\smallskip

The existence of a solution to the boundary value problem~\eqref{prob}, in  Theorem~\ref{thm:2}, will be established by arguing that energy functional attains a minimum in the energy space. For this, we adapt the strategy used in~\cite[Theorem 2.10]{radu}.\medskip

As the boundary value problem~\eqref{prob} has variational structure, we introduce the associated energy functional $\E:H^s_0(\Omega)\to\R$ as follows
\begin{equation}\label{eq:energy}
    \E(u):= \int_{\Omega} \Gam{\frac{|\nabla^{s} u|^2}{2}}dx - \int_{\Omega} F(u)dx -\int_{\Omega}hu\,dx,
\end{equation}
where $F(t):= \int\limits_{0}^{t} f(\tau)\,d\tau$. One can readily check that $\mathcal{J}$ is Frech\^{e}t-differentible in $H_0^s(\Omega)$, with derivative $\E(u)\in H^{-s}_0(\Omega)$ equal to
\begin{equation}\label{eq:energy:derivative}
    \E'(u)[\varphi] = \int_{\Omega} \gam{\frac{|\nabla^{s} u|^2}{2}}\inn{\nabla^{s} u}{\nabla^{s}\varphi}dx - \int_{\Omega} f(u)\varphi dx -\int_{\Omega} h\varphi\,dx,
\end{equation}
for any $\varphi\in H_0^s(\Omega)$.
In order to adapt the techniques from~\cite{radu}, we observe that it is not possible to define functionals \(J\) and \(J^{-1}\) exactly as in~\cite{radu}, as the operator under consideration is quasilinear. Nevertheless, we overcome the issue by noting that our operators are equivalent to those associated with \cite{radu}. For this purpose, and henceforth, we denote 
\[
    J:H^s_0(\Omega)\to\R,\quad J(\varphi):= \int_{\Omega} \gam{\frac{|\nabla^{s} u|^2}{2}}\inn{\nabla^{s} u}{\nabla^{s}\varphi}dx,
\]
where $u\in H^s_0(\Omega)$ has been fixed a priori (and it will be clear from the context). We also let $\bar{J}=J^{-1}$.
\section{Existence of a local minimum of $\E$}\label{sec2}
In this section, we prove the existence of a solution for Problem~\eqref{prob}, via minimization.
\smallskip 

\begin{lemma}\label{lemwsl}
    $\E$ is weakly sequentially lower semi-continuous over $H^s_0(\Omega)$.
\end{lemma}
\begin{proof}
It is sufficient to prove that 
\begin{equation}\label{eq:def:Phi}
    \Phi (u):= \int_{\Omega} \Gam{\frac{|\nabla^{s} u|^2}{2}}dx, \quad u\in H_{0}^s(\Omega).   
\end{equation}
is weakly sequentially lower semi-continuous. Indeed, the convexity of $g(z):=\Gamma (z^2)$ (see~\ref{item:gamma:2}) guarantees
\begin{eqnarray*}
      \Phi(u_1)-\Phi (u_2) &=& \int_{\Omega} g\left(\frac{|\nabla^{s} u_1|}{\sqrt{2}}\right)dx - \int_{\Omega} g\left(\frac{|\nabla^{s} u_2|}{\sqrt{2}}\right)dx\\[0.25em] 
      &\geq& \int_{\Omega} g'\left(\frac{|\nabla^{s} u_2|}{\sqrt{2}}\right) \left(\frac{|\nabla^{s} u_1|}{\sqrt{2}}-\frac{|\nabla^{s} u_2|}{\sqrt{2}}\right) dx\\[0.25em]
      &\geq& \int_{\Omega} \gamma\left(\frac{|\nabla^{s} u_2|^2}{2}\right)\sqrt{2}|\nabla^{s} u_2 |\left(\frac{|\nabla^{s} u_1|}{\sqrt{2}}-\frac{|\nabla^{s} u_2|}{\sqrt{2}} \right) dx.
\end{eqnarray*}
for every $u_1,u_2 \in H_0^s(\Omega)$. In particular, in light of~\eqref{eq:energy:derivative}, we see
\begin{align}
    \Phi (u_1)- \Phi (u_2) - & \Phi'(u_2)[u_1-u_2]\nonumber\\[0.25em]
    &\geq\int_{\Omega} \gam{\frac{|\nabla^{s} u_2|^2}{2}}\left(|\nabla^{s} u_2|(|\nabla^{s} u_1|-|\nabla^{s} u_2|)-\nabla^{s} u_2\cdot\nabla^{s}(u_1-u_2) \right)dx\nonumber\\
    &\geq 0. \label{eqpo}
\end{align}
In other words, 
\begin{equation}\label{eq:lemma:lowesci:aux:1}
    \Phi(u_2)+\Phi'(u_2)[u_1-u_2]\leq \Phi(u_1).    
\end{equation}
Given now any $w\in H^s_0(\Omega)$ and a sequence $(w_n)_n\subset H_0^s(\Omega)$ with $w_n \rightharpoonup w$ weakly, choose $u_2=w$ and $u_1=w_n$ in~\eqref{eq:lemma:lowesci:aux:1}. Since $\lim_{n \to \infty} \Phi'(w)[w_n-w]=0$, we conclude the desired semi-continuity
\[
    \Phi(w) \leq \liminf\limits_{n \to \infty} \Phi(w_n). 
\]
\end{proof}

\subsection{The asymptotic sublinear growth case}
Recall that the assumption $f(t)=\nu g(t)$ implies that $F(t)=\nu G(t)$, in the definition of the energy~\eqref{eq:energy}.\smallskip 

\begin{lemma}\label{lemcoercive}
    Assume~\ref{item:g:1},\ref{item:g:2},\ref{item:gamma:1}, and that $\|h\|_{L^2(\Omega)}$ is sufficiently small. Then, $\E$ is coercive and bounded from below on $X\cap B$, for any ball $B$ in $H^s_0(\Omega)$.
\end{lemma}
\begin{proof}
For $\nu$ fixed, there exists $A>0$ such that
\[
    \forall t\in\R: \quad \nu G(t) \leq \poincConst\frac{\gmin}{4}\,t^2 +A,
\]
 given that $G(t)=o(t^2)$ as $t\to \pm \infty$ (see~\ref{item:g:1}). Also, $\Gam{t}\geq \gmin t$, from~\ref{item:gamma:1}. Hence, for any $u\in X$,
\begin{align*}
  \E (u) 
  &\geq \frac{\gmin}{2}\int_{\Omega} |\nabla^{s} u|^2 dx -\poincConst\frac{\gmin}{4} \int_{\Omega} u^2 dx - A|\Omega| - \|h\|_{L^2(\Omega)}\|u\|_{L^2(\Omega)}  \\
  &=\frac{\gmin}{2}\|u\|^2_{H^s_0(\Omega)}-\poincConst\frac{\gmin}{4}\|u\|^2_{L^2(\Omega)} - A|\Omega| - \|h\|_{L^2(\Omega)}\|u\|^2_{L^2(\Omega)}  \\
  &\geq \poincConst\frac{\gmin}{4}\|u\|^2_{L^2(\Omega)} -\|h\|_{L^2(\Omega)}\|u\|_{L^2(\Omega)}- A|\Omega|,
\end{align*}
from which follows from the coercivity and boundedness from below of $\E$.
\end{proof}

\begin{lemma}\label{leminfv}
    Assume~\ref{item:g:2} and $h\equiv 0$. Then, there exists $\nu_1>0$ so that for all $\nu>\nu_1$ 
    \[
    \inf_{u\in X} \E(u) < 0.
    \]
\end{lemma}
\begin{proof}
Take $t_0>0$ as in~\ref{item:g:2}. By Tietze’s extension theorem, for any compact set \( K\subset\subset \Omega \) there is a map $w_K \in X\cap C(\overline{\Omega})$ such that $w_K\equiv t_0$ in $K$ and $|w_K| \leq t_0$ in $\Omega$. Whence,
\[
\int_{\Omega} G(w_K) \, dx \geq |K|\,G(t_0)-|\Omega\setminus K|(\max_{s \in [-t_0, t_0]} |G(s)|).
\]
Fixing now $K$ in such a way that $|\Omega\setminus K|\leq \min\{|\Omega|G(t_0)/(4\|G\|_{C([-t_0,t_0])}),|\Omega|/2\}$, we deduce
\[
\int_{\Omega} G(w_K) \, dx\geq \frac{|\Omega|}{4} G(t_0)> 0.
\]
This, combined with~\ref{item:gamma:1}, allows us to conclude that, for \( \nu \) sufficiently large, there holds:  
\[
    \E(w_K) \leq \gmax\int_{\Omega}|\nabla^s w_K|^2 \, dx - \nu \int_{\Omega} G(w_K) \, dx < 0.
\]
\end{proof}

\begin{lemma}\label{nosol}
    Assume~\ref{item:g:3},\ref{item:gamma:1}, and $h\equiv 0$. Then, for $\nu >0$ sufficiently small, Problem~\eqref{prob} does not admit nontrivial solutions. 
\end{lemma}
\begin{proof}
Let \( \nu > 0 \) be such that Problem~\eqref{prob} has a solution $u\in X$. Testing the equation with $u$, and using~\ref{item:gamma:1}, we obtain
\[
    \frac{\gmin}{2}\int_{\Omega} |\nabla^{s} u|^2 dx\leq \int_{\Omega} \gam{\frac{|\nabla^{s} u|^2}{2}}|\nabla^{s} u|^2 dx=\nu \int_{\Omega} g(u)u \, dx \leq \nu \int_{\Omega} |g(u)||u|\, dx. 
\]
Now, in view of~\ref{item:g:3}, there exists $C>0$ so that 
\[
    \int_{\Omega} |g(u)||u|\, dx\leq C\int_{\Omega} |u|^2\, dx\leq C(\poincConst)^{-1}\|u\|_{H^s_0(\Omega)}^2.
\]
Putting together the two inequalities above, we deduce
\[
    \left(\frac{\gmin}{2} -\nu \frac{C}{\poincConst}\right) \|u\|_{H^s_0(\Omega)}^2 \leq 0.
\]
We conclude that $u\equiv 0$ is the only possible solution, provided $\nu>0$ small enough. 
\end{proof}

\subsection{Proof of Theorem~\ref{thm:2}}
    For part (i), observe that \(\E \) is bounded,  coercive, and weakly sequentially lower semicontinuous by Lemmas \ref{lemwsl} and \ref{lemcoercive}. We will establish the boundedness condition of $\E$ by contradiction. Suppose that there exists some $u\in X$ with $\|u\|_{H^s_0(\Omega)}=R$ such that $\E'+\rho\E=0$ for some $\rho>0$. Then,
    \[
        \int_{\Omega}\gam{\frac{|\nabla^su|^2}{2}}|\nabla^su|^2dx-\int_{\Omega}\nu g(u)u\,dx-\int_{\Omega}hu\,dx+\rho\int_{\Omega}\gam{\frac{|\nabla^su|^2}{2}}|\nabla^su|^2\,dx=0.
    \]
    From here, we see that
    \[
        (1+\rho)\gmin\|u\|^2_{H^s_0(\Omega)}\leq \nu\int_{\Omega}g(u)u\,dx+\int_{\Omega}hu\,dx
    \]
    Whence, H\"older's inequality together with \ref{item:g:3}~for $R\geq t_0$ imply
    \[
        R^2<(1+\rho)\gmin R^2\leq CR^2+\|h\|^2_{L^2(\Omega)}R,
    \]
    thus contradicting~\eqref{condbound}. Thus, by \cite[Theorem 2.10]{radu}, it is guaranteed that $\E$ admits a global minimum. Furthermore, $u$ is nontrivial since $h\neq 0$.
    
    For part (ii.a), the boundedness condition is also satisfied when $h\equiv 0$. Moreover, from~Lemmas~\ref{lemwsl} and~\ref{lemcoercive}, together with \cite[Theorem 2.10]{radu}, yield that the functional $\E$ attains a global minimum $u$. In fact, $u$ is nontrivial for $\nu >0$ large enough, in light of~Lemma~\ref{leminfv}. Finally, (ii.b) follows directly by Lemma~\ref{nosol}. This concludes the proof of Theorem~\ref{thm:2}.\qed

\subsection{The linear growth case}
\begin{lemma}\label{lc}
  Assume conditions~\ref{item:f:1}-\ref{item:f:2} and~\ref{item:gamma:1}-\ref{item:gamma:2}. For every $h\in L^2(\Omega)$ with \( \|h\|_{L^2(\Omega)} \) sufficiently small, there exists \( r_h > 0 \) such that 
  \[
    \inf\{\E(u):\|u\|_{H_0^s(\Omega)} = r_h \}\geq \alpha(r_h)>0.
  \]  
\end{lemma}
\begin{proof}
We begin by observing that, in light of~\ref{item:f:1}-\ref{item:f:2}, there exists $C_0>1$ large enough, such that
\[
    |f(t)|\leq \frac{\gmin}{2}\poincConst\,|t|+C_0|t|^p,\quad\text{ for all }t\in\R.
\]
This, in turn, implies
\[
    |F(t)|\leq \frac{\gmin}{4}\poincConst\,t^2+\frac{C_0}{p+1}|t|^p,\quad\text{ in }\R.
\]
Now, if $u\in X$, from the definition of the energy~\eqref{eq:energy} and~\ref{item:gamma:1}, there follows that
\[
\begin{aligned}
    \E(u)
    &\geq \frac{\gmin}{2}\int_{\Omega}|\nabla^s u|^2dx-\frac{\gmin}{4}\poincConst\int_{\Omega}u^2dx-\frac{C_0}{p+1}\int_{\Omega}|u|^{p+1}dx-\int_{\Omega}hu\,dx \\[0.25em]
    &\geq \frac{\gmin}{2}\|u\|^2_{H^s_0(\Omega)}-\frac{\gmin}{4}\poincConst\|u\|^2_{L^2(\Omega)}-\frac{C_0}{p+1}\|u\|^{p+1}_{L^{p+1}(\Omega)}-\|h\|_{L^2(\Omega)}\|u\|_{L^2(\Omega)}\\[0.25em]
    &\geq \frac{\gmin}{4}\|u\|^2_{H^s_0(\Omega)}-\frac{C_0}{p+1}\,C^{p+1}_s(\Omega)\|u\|^{p+1}_{H^s_0(\Omega)}-\sqrt{\poincConst}\|h\|_{L^2(\Omega)}\|u\|_{H^s_0(\Omega)},
\end{aligned}
\]
where we have used the variational characterization of the first eigenvalue $\poincConst$ of $(-\Delta)^s$ on $\Omega$, as well as the fractional Sobolev embedding $H^s_0(\Omega)\hookrightarrow L^{p+1}(\Omega)$ with Sobolev constant $C_s(\Omega)$, using $p+1\in (2,2^*_s)$. In other words, we obtained that
\[
    \E(u)\geq \bigl(\frac{\gmin}{4}-a\|u\|^{p-1}_{H^s_0(\Omega)}-\frac{b}{\|u\|_{L^2(\Omega)}}\bigr) \|u\|^2_{H^s_0(\Omega)},
\]
for the constants $a:=C_0C^p_s(\Omega)/(p+1)$ and $b:=\sqrt{\poincConst}\|h\|_{L^2(\Omega)}$. Let us  denote $r:=\|u\|_{H^s_0(\Omega)}>0$, and observe that 
$r\mapsto ar^{p-1}+b/r=:q(r)$ is smooth for $r>0$, strictly convex as $p>2$, $\lim_{r\to0^+}q(r)=+\infty$ and $\lim_{r\to\infty}q(r)=+\infty$. Its only global minimum is attained at $r_{\star}=(b/(a(p-1)))^{1/p}$ with $q(r_{\star})=O(b^{p-1})$. Therefore, taking $\|h\|_{L^2(\Omega)}=O(b)$ small enough to make $q(r_{\star})<\gmin/4$, we conclude the desired inequality $\inf\{\E(u):\|u\|_{H^s_0(\Omega)}=r\}\geq \alpha(r)>0$, with $\alpha(r)=\gamma/4-q(r_{\star})$.
\end{proof}

\begin{proposition}\label{lem1sol}
Assume~\ref{item:f:1}-\ref{item:f:2} and~\ref{item:gamma:1}-\ref{item:gamma:2}, and that $\|h\|_{L^2(\Omega)}$ is sufficiently small. Then, Problem~\eqref{prob} admits a solution which is a local minimizer of $\E$. In addition, this solution is nontrivial, provided $h\neq 0$. 
\end{proposition}
\begin{proof}
Take both $h\in L^2(\Omega)$ and $r_h>0$ as in Lemma~\ref{lc}. The problem \( m := \inf\{\E(u): \|u\|\leq r_h\} \) admits a minimizer $u_{\star}$ in view of Lemma~\ref{lemwsl}. Indeed, since $\overline{B(0,r_h)}\in X$ is bounded and weakly closed in $H^s_0(\Omega)$, we can use the direct method of the calculus of variations to obtain a cluster point $u_{\star}\in\overline{B(0,r_h)}\cap X$ from any minimizing sequence, say \((u_n)_{n \in \mathbb{N}} \subset \overline{B(0,r_h)}\cap X\). As the energy is weakly l.s.c., it follows that \( m=\E(u_{\star})\leq\liminf_{n\to\infty}\E(u_n)=m\). Furthermore, the minimizer $u_{\star}$ lies in the interior of $B(0,r_h)\cap X$, due to the fact that
$m\leq \E(0) = 0 < \alpha(r_h)$. Thus, \( u_{\star} \) is a critical point of \(\E\), rendering a solution to~\eqref{prob}. Moreover, if \( h \ne 0 \), then \( u_{\star} \) is nontrivial.
\end{proof}

%%%%%%%%%%%%%%%%%%%%%%%%%%%%%%%%%%%%%%%%%%%%%%%%
\begin{proposition}\label{lem1sol}
Assume that~\ref{item:f:1}-\ref{item:f:2} and~\ref{item:gamma:1}-\ref{item:gamma:2} are satisfied, and that $\|h\|_{L^2(\Omega)}$ is sufficiently small. Then Problem~\eqref{prob} admits a solution which is a local minimizer of $\E$. In addition, this solution is nontrivial if $h\ne 0$. 
\end{proposition}

\begin{proof}
Take both $h\in L^2(\Omega)$ and $r_h>0$ as in Lemma~\ref{lc}. The problem \( m := \inf\{\E(u): \|u\|\leq r_h\} \) admits a minimizer $u_{\star}$ in view of Lemma~\ref{lemwsl}. Indeed, since $\overline{B(0,r_h)}\cap X$ is bounded and weakly closed in $H^s_0(\Omega)$, we can use the direct method of the calculus of variations to obtain a cluster point $u_{\star}\in\overline{B(0,r_h)}\cap X$ from any minimizing sequence, say \((u_n)_{n \in \mathbb{N}} \subset \overline{B(0,r_h)}\cap X\). As the energy is weakly l.s.c., it follows that \( m=\E(u_{\star})\leq\liminf_{n\to\infty}\E(u_n)=m\). Furthermore, the minimizer $u_{\star}$ lies in the interior of $B(0,r_h)\cap X$, due to the fact that
$m\leq \E(0) = 0 < \alpha(r_h)$. Thus, \( u_{\star} \) is a critical point of \(\E\), rendering a solution to~\eqref{prob}. Moreover, if \( h \ne 0 \), then \( u_{\star} \) is nontrivial.
\end{proof}

\section{A mountain pass solution for Problem~\eqref{prob}}
In this section, we prove the existence of a solution to the boundary value problem~\eqref{prob} using the mountain pass theorem. This entails arguing that the energy $\E$ satisfies the geometric conditions required by the mountain pass theorem, which will be a consequence of  Lemma~\ref{lc} together with the auxiliary result.\smallskip

\begin{lemma}\label{cond2}
    Assume~\ref{item:f:2}-\ref{item:f:3}, and let $\varphi_1$ be the first eigenfunction of $(-\Delta)^s$ in $H^s_0(\Omega)$. Then $\E (t\varphi_1)<0$ for all sufficiently large  $t\in \R$.
\end{lemma}

\begin{proof}
Let us obtain a bound for the first term in the energy $\E$. In view of~\ref{item:gamma:1}, for all $n\in \N$, $|\gamma(\tau)-\gamma(\infty)| < 1/n$, provided $\tau\geq \tau_n>>1$. Whence, for  $x\in\Omega$,
\[
\begin{aligned}
    \lim\limits_{t \to \infty}\frac{1}{t^2}\Gamma\left(\dfrac{t^2}{2}|\nabla^s \varphi_1(x)|^2\right)
    &= \lim\limits_{t\to\infty}\dfrac{1}{t^2}\left( \int\limits_{0}^{t^2|\nabla^s \varphi_1(x)|^2/2}(\gamma(\tau)-\gamma(\infty))d\tau+\int\limits_{0}^{t^2|\nabla^s \varphi_1(x)|^2/2} \gamma(\infty)d\tau\right)\\
    &= \dfrac{\gamma(\infty)}{2} |\nabla^s \varphi_1(x)|^2.
\end{aligned}
\]
On the other hand, the boundedness of $\gamma$ implies,
\[
    \forall t\geq 1: \quad\dfrac{1}{t^2}\Gamma\left(\dfrac{t^2}{2}|\nabla^s \varphi_1(x)|^2\right) = \dfrac{1}{t^2}\int\limits_{0}^{t^2|\nabla^s \varphi_1(x)|^2/2} \gamma(\tau)d\tau \leq \frac{\gmax}{2}|\nabla^s \varphi_1(x)|^2.
\]
Whence, the dominated convergence theorem yields
\[
    \lim\limits_{t \to \infty} \int_{\Omega}\dfrac{1}{t^2}\Gamma\left(\dfrac{t^2}{2}|\nabla^s \varphi_1(x)|^2\right)dx 
    = \dfrac{\gamma(\infty)}{2}\int_{\Omega}|\nabla^s \varphi_1(x)|^2dx
    =\dfrac{\gamma(\infty)}{2}\poincConst\int_{\Omega}| \varphi_1|^2.
\]
The second term in the energy is bounded by observing that there exists $\delta>0$ such that $F(t) \geq \tfrac{\gamma(\infty)}{2}(\poincConst + \delta)t^2$, for any $t\geq 1$ (from~\ref{item:f:3}). Thus,
\[
    \liminf\limits_{t \to \infty}  \dfrac{1}{t^2}  \int\limits_{\Omega} F(t \varphi_1) dx \geq \dfrac{\gamma(\infty)}{2}(\poincConst + \delta) \int\limits_{\Omega} \varphi_1^2(x)dx,
\]
Putting together these two bounds, 
\[
\begin{aligned}
  \limsup\limits_{t \to \infty} \dfrac{1}{t^2}\E (t\varphi_1)&=\limsup\limits_{t \to \infty}\left( \dfrac{1}{t^2}\int\limits_{\Omega}\Gamma\left(\dfrac{t^2}{2}|\nabla^s \varphi_1|^2\right) - \dfrac{1}{t^2}  \int\limits_{\Omega} F(t \varphi_1) dx  -\dfrac{1}{t^2} \int\limits_{\Omega} h(t\varphi_1)dx  \right)\\ 
  &\leq -\dfrac{\delta\gamma(\infty)}{2}  \int\limits_{\Omega} \varphi_1^2dx<0.    
\end{aligned}
\]
\end{proof}

We now study the compactness of the functional \( \E \), by establishing the so-called Schechter-Palais-Smale condition. Let us recall that a sequence \( (u_n)_{n \in \mathbb{N}} \subset X_R \setminus \{0\} \) is said to be a \emph{Schechter-Palais-Smale sequence for \( \E \) at level \( c \in \mathbb{R} \)}, where \( X_R \) is the closed ball of radius \( R \) centered at the origin in \( H^s_0(\Omega) \), if
\[
    \E(u_n)\xrightarrow[n\to\infty]{} c,
\]
and, moreover, if it satisfies any the following conditions\smallskip

\begin{itemize}
    \item[(a)] \( u_n \in X_R \) for all \( n \), and 
    \begin{equation}\label{cond1}
        \E'(u_n)\xrightarrow[n\to\infty]{} 0; 
    \end{equation}
    
    \item[(b)] \( u_n \in \partial X_R \) for all \( n \), and 
    \begin{equation}\label{cond2}
        \E'(u_n) - \dfrac{\langle \E'(u_n), u_n\rangle}{R^2} J u_n \xrightarrow[n\to\infty]{} 0,\;\text{ while }\; \langle \E'(u_n), u_n \rangle \leq 0;
    \end{equation}

    \item[(c)] \( u_n \in \partial X_R \) for all \( n \), and 
    \begin{equation}\label{cond3} 
        \bar{J} \E'(u_n) - \dfrac{\langle \E u_n, \bar{J} \E'(u_n) \rangle}{R^2} u_n\xrightarrow[n\to\infty]{} 0,\;\text{ while }\; \langle Ju_n, \bar{J} \E'(u_n) \rangle \leq 0.
    \end{equation}
\end{itemize}

The following results will end up establishing that, in fact, \(\E\) satisfies the Schechter-Palais-Smale condition. The core of the argument relies in establishing the convergence of bounded Schechter-Palais-Smale sequences for \(\E\).

First, we prove that the following conditions are satisfied

\[
u - \bar{J}\E'(u)\in X \quad \text{for all} \ u \in X,
\]
\[
\min \left\{ \langle \E'(u), u\rangle,\langle Ju, \bar{J}\E'(u) \rangle \right\} \geq -\nu_0 \quad \text{for all} \ u \in \partial X_R, \;\text{ for some }\; \nu_0>0.
\]
Indeed,
\[
    u - \bar{J}\E'(u) \geq 0  \quad\Leftrightarrow\quad u \geq \dfrac{1}{\gmin}\dfrac{\|E'(u)\|}{\|u\|_{H_0^s(\Omega)}}.
\]
Using \ref{item:f:4} and $\|h\|_{L^2(\Omega)}$ is small we obtain 
\[
    u \geq - C(\|u\|_{H_0^s(\Omega)} +1).
\]
Now, by Lemma \ref{lemcoercive} we have $\langle \E'(u), u\rangle \geq -\nu_0$ for some $\nu_0>0$. Moreover, 
\[
\langle Ju, \bar{J}\E'(u) \rangle \geq - \|J(u)\|_{H_0^{s}}\|\bar{J}\E'(u)\|_{H_0^{-s}}\geq -C\|J(u)\|_{H_0^{s}}\|\E'(u)\|_{H_0^{-s}}
\]
Since \( J(u) \) and \( \E'(u) \) are bounded, the condition is satisfied.\smallskip

\begin{lemma}\label{lemconver}
    Assume~\ref{item:gamma:1}-\ref{item:gamma:2}. Suppose that $(u_n)_{n\in\N} \subset X$ is so that $u_n \rightharpoonup u$ in $X$, $\nabla^s u_n \to \nabla^s u $ a.e. on $\Omega$, and
    \[
        \limsup\limits_{n\to \infty} \int_{\Omega}\Gam{\frac{|\nabla^s u_n|}{2}}dx\leq \int_{\Omega}\Gam{\frac{|\nabla^s u|}{2}}dx.
    \]
    %where $\Phi$ was defined in~\eqref{eq:def:Phi}. 
    Then, $u_n \to u$ in $X$, up to a subsequence.
\end{lemma}

\begin{proof}
We follow~\cite[Theorem 2]{BrezisLieb} adapted to our case. Define, for $k>1$ and $\epsilon\in(0,k)$ fixed,
\begin{align*}
    W_{n,\epsilon}(x)&:= \Gam{\frac{|\nabla^{s} u_n|^2}{2}}- \Gam{\frac{|\nabla^{s} u|^2}{2}}- 
    \Gam{\frac{|\nabla^{s} u_n|^2}{2}- \frac{|\nabla^{s} u|^2}{2}} \\ 
    &\quad\quad\quad - \epsilon \left[\Gam{k\left(\frac{|\nabla^{s} u_n|^2}{2}- \frac{|\nabla^{s} u|^2}{2}\right)}-k\Gam{\frac{|\nabla^{s} u_n|^2}{2}- \frac{|\nabla^{s} u|^2}{2} }\right]
\end{align*}
Clearly, $W_n \to 0$ a.e. in $\Omega$. There also exists domination in view of~\cite[Lemma 3]{BrezisLieb} with $j(s):=\Gam{t^2}$,
\[
|W_{n,\epsilon}| \leq 2\, \Gam{\frac{C_{\epsilon}^2|\nabla^{s}u_n|^2}{2}} \in L^1(\Omega).
\]
Thus, the Dominated Convergence Theorem ensures 
\[
    \int_{\Omega} W_{n, \epsilon}(x)\, dx\xrightarrow[n \to \infty]{} 0, 
\]
Now, the convexity of $j$ implies that
\[
    \limsup_{n\to \infty} \left\{\int_{\Omega} \Gam{\frac{|\nabla^{s} u_n|^2}{2}}dx - \int_{\Omega} \Gam{\frac{|\nabla^{s}u|^2}{2}} dx -  \int_{\Omega}\Gam{\frac{|\nabla^{s} u_n|^2}{2}- \frac{|\nabla^{s} u|^2}{2}} dx\right\}\leq \epsilon\,C
\]
whence
\[
    \int_{\Omega} \Gam{\frac{|\nabla^{s} u_n|^2}{2}}dx - \int_{\Omega} \Gam{\frac{|\nabla^{s} u|^2}{2}} dx - \int_{\Omega}\Gam{\frac{|\nabla^{s} u_n|^2}{2}- \frac{|\nabla^{s} u|^2}{2}} dx\longrightarrow 0.
\]
The latter, added to the hypothesis on the limit superior, allows us to conclude
\[
    \int_{\Omega} \Gam{\frac{|\nabla^{s} u_n|^2}{2}- \frac{|\nabla^{s} u|^2}{2}} dx \to 0.
\]
Finally, the bound on the integrand, $\Gam{t}\geq \gamma_{\min} t$ shows $u_n \to u$ in $X$, up to subsequence.
\end{proof}

The following is a technical result, required to apply the algo property in~\cite{Landes}.\smallskip

\begin{lemma}\label{monotono}
  Under~\ref{item:gamma:1}-\ref{item:gamma:2},  $\beta(z):=\gamma\left(\tfrac{|z|^2}{2}\right)z$ is a strictly monotone vector-field in $\R^d$. 
\end{lemma}

\begin{proof}
 Let $z,z'\in\R^d$ with $z\neq z'$. Since the inner product is bilinear, without loss of generality we assume $|z|\geq |z'|$. Furthermore, we assume $|z'|>0$, as otherwise the property follows:
 \[
   \inn{\gam{\frac{|z|^2}{2}}z - \gam{\frac{|0|^2}{2}}0}{z-0}= \gam{\frac{|z|^2}{2}}|z|^2>0.
 \]
 The monotonicity of $t\mapsto\tfrac{d}{dt}\Gamma(t^2)=2\gamma(t^2)t$ (see~\ref{item:gamma:2}) together with the Cauchy-Schwarz inequality, yields
\[
\begin{aligned}
    \inn{\gam{\frac{|z|^2}{2}}z - \gam{\frac{|z'|^2}{2}}z'}{z-z'} 
    &=\gam{\frac{|z|^2}{2}}(|z|^2-\inn{z}{z'})-\gam{\frac{|z'|^2}{2}}(\inn{z'}{z}-|z'|^2)\\
    &\geq \gam{\frac{|z'|^2}{2}}|z'|\left(|z|-\frac{\inn{z}{z'}}{|z|}\right)-\gam{\frac{|z'|^2}{2}}|z'|\left(\frac{\inn{z'}{z}}{|z'|}-|z'|\right)\\
    &=\gam{\frac{|z'|^2}{2}}|z'|\left(|z|-\frac{\inn{z'}{z}}{|z'|}-\frac{\inn{z}{z'}}{|z|}+|z'|\right)\geq 0.
\end{aligned}
\]
\end{proof}

\begin{lemma}\label{lemconverg}
Assume~\ref{item:f:1}-\ref{item:f:2}  and~\ref{item:gamma:1}-\ref{item:gamma:2}. Then, each bounded Schechter-Palais-Smale sequence of $\E$ admits a convergent subsequence in $X$.   
\end{lemma}

\begin{proof}
Let $(u_n)_{n\in \N}\subset X$ be a bounded Schechter-Palais-Smale sequence of $\E$. Assume, without loss of generality,
\[
    u_n \rightharpoonup u\,\text{ in }X\quad\text{ and }\quad u_n \to u \;\text{ in }L^q(\Omega)\;\;\text{ for }\;q \in [2, 2^*).
\]
If $(u_n)_{n\in \N}$ satisfy condition~\eqref{cond2}, we can assume that $-\langle \E'(u_k), u_n \rangle/R^2\to \rho \geq 0$, thus
\[
    \E'(u_n) + \rho \Phi' (u_n) \to 0.
\]
Since \(\E'(u_n)(u - u_n) \to 0\), we directly argue that \(\Phi'(u_n)(u - u_n) \to 0\). Hence, 
\[
    (\Phi'(u_n) - \Phi'(u))(u_n - u) = \int_{\Omega} \inn{\gam{\frac{|\nabla^{s} u_n|^2}{2}}\nabla^{s} u_n - \gam{\frac{|\nabla^{s} u|^2}{2}}\nabla^{s} u}{\nabla^s u_n -\nabla^s u}dx\longrightarrow 0.
\]
Consequently,
\[
    \inn{\gam{\frac{|\nabla^{s} u_n|^2}{2}}\nabla^{s} u_n - \gam{\frac{|\nabla^{s} u|^2}{2}}\nabla^{s} u}{\nabla^s u_n - \nabla^s u}\xrightarrow[n\to\infty]{}0\quad\text{ a.e. }\; x \in \Omega.
\]
The latter, together with Lemma \ref{monotono}, meet all the hypotheses of~\cite[Lemma 6]{Landes}. Thus,
\[
    \nabla^{s} u_n \to \nabla^{s} u\quad\text{ a.e. }\;\;\text{in}\;\;\Omega.
\]
We now recall the convexity character of $\Phi$,~\eqref{eqpo}, that yields 
\[
\Phi(u) - \Phi(u_n) - \Phi'(u_n)[u-u_n]\geq 0. 
\]
which, in turn, implies
\[
\limsup_{n \to \infty} \Phi(u_n) \leq \Phi (u).
\]
As all the hypotheses of Lemma~\ref{lemconver} are met, the desired convergence $u_{n_k}\to u$ in $X$ is obtained.\\

If $(u_n)_{n\in \N}$ satisfy condition~\eqref{cond3}, we assume that $-\langle \Phi'(u_n), \bar{J} \E'(u_n) \rangle/R^2\to \rho \geq 0$, then 
\[
 \bar{J} \E'(u_n) + \rho u_n \to 0.
\]
When $\rho=0$, it follows that \(\E'(u_n)(u - u_n) \to 0\), and proceed in a similar way to the previous case. If $\rho>0$, taking
\[
    v_n:=  \bar{J} \E'(u_n) + \rho u_n ,
\]
then,
\[
    u_n= \dfrac{1}{\rho}v_n - \dfrac{1}{\rho} \bar{J} \E'(u_n) .
\]
This is, 
\[
    u_n= \dfrac{1}{\rho}v_n - \dfrac{1}{\rho} \bar{J}(\Phi'(u_n) - f(u_n) - h ).
\]

By Lemma \ref{lemconver} and \ref{item:gamma:1}, it follows that the sequence \( \Phi'(u_n) \) is relatively compact in \( H^{-s}_0(\Omega) \). Furthermore, the relative compactness of the sequence \( f(u_n) - h \) in \( H^{-s}_0(\Omega) \) is also clear.
\end{proof}
\smallskip

\begin{proposition}\label{cotau}
    The sequence $(u_n)_{n\in \N}$ is bounded in $X$.
\end{proposition}

\begin{proof}
The criticality of $u_n$ with respect to $\E$, tested against $u_n$, gives 
\begin{equation}\label{deriEu}
     \int_{\Omega} \gam{\frac{|\nabla^{s} u_n|^2}{2}}|\nabla^{s} u_n|^2dx -\int_{\Omega} f(u_n)u_n dx-\int_{\Omega} hu_n\,dx=0,  
\end{equation}
whence,
\[
    \gamma_{\min}\int_{\Omega} |\nabla^{s} u_n|^2dx\leq  \int_{\Omega}f(u_n) u_n dx+\int_{\Omega} hu_n\,dx
\]
in view of~\ref{item:gamma:1}. This shows that the boundedness of $(u_n)_{n\in \N}$ in $X$ is reduced to proving the boundedness of the sequence in $L^2(\Omega)$. 

To establish the latter, we argue by contradiction, so let us assume \(\|u_n\|_{L^2(\Omega)} \to \infty\). We define an auxiliary sequence $v_n:=u_n/\|u_n\|_{L^2(\Omega)}$ that clearly lies in the unit sphere in $L^2(\Omega)$. We now claim that $(v_n)_{n\in \N}$ is, in fact, bounded in $X$. On the one hand, equation~\eqref{deriEu} normalized by $\|u_n\|_{L^2(\Omega)}^2$ implies
\begin{align*}
    \int_{\Omega} \gam{\frac{|\nabla^{s} u_n|^2}{2}}|\nabla^{s} v_n|^2dx 
    &\leq \|f(u_n)\|_{L^2(\Omega)}\|u_n\|^{-1}_{L^2(\Omega)} + \|h\|_{L^2(\Omega)}\|u_n\|^{-1}_{L^2(\Omega)}.
\end{align*}
On the other hand, assumption~\ref{item:f:4} yields the bound
\[
 \|f(u_n) \|_{L^2(\Omega)} \leq C(1+ \|u_n\|_{L^2(\Omega)}).
\]
The latter inequalities combined with~\ref{item:gamma:1} finish the proof of the claim. Hence, up to a subsequence 
\[
    \exists v\in X:\quad v_n \rightharpoonup v \;\;\text{in}\; H_0^s(\Omega) \quad \text{and} \quad v_n \to v \;\;\text{in}\; L^2(\Omega).
\]
Now, we can test the criticality of $u_n$ in $\E$ against $\varphi_1/\|u_n\|_{L^2(\Omega)}$, to obtain
\begin{equation}\label{testv}
\begin{split}
    &\int_{\Omega} \gam{\frac{|\nabla^{s} u_n|^2}{2}}\inn{\nabla^{s}v_n}{\nabla^s\varphi_1}dx 
    =\int_{\Omega} f(u_n)u_n\dfrac{\varphi_1}{\|u_n\|_{L^2(\Omega)}}dx+\int_{\Omega} h\dfrac{\varphi_1}{\|u_n\|_{L^2(\Omega)}}\,dx\\
    &\geq  A \int_{\Omega} v_n \varphi_1\, dx -\dfrac{B}{\|u_n\|_{L^2(\Omega)}}\int_{\Omega} \varphi_1 \,dx+\dfrac{1}{\|u_n\|_{L^2(\Omega)}}\int_{\Omega} h\varphi_1\,dx,
\end{split}
\end{equation}
for some constants $A>\gamma(\infty)\lambda_1^s$ and $B>0$ and $n\in\N$ asymptotically large, in view of~\ref{item:f:3}, since 
\[
    f(t) \geq A t - B,\;\;\text{for all}\; t\in \R.
\]
Taking limit $n\to\infty$ in~\eqref{testv}, we get:
\[
    \liminf_{n\to\infty}\int_{\Omega} \gam{\frac{|\nabla^{s} u_n|^2}{2}}\inn{\nabla^{s}v_n}{\nabla^s\varphi_1}dx\geq A\int_{\Omega} v \varphi_1\,dx.
\]
Now, based on the assumptions in~\ref{item:gamma:1} of the density, it can be argued that
\[
    \lim_{n\to\infty}\int_{\Omega} \gam{\frac{|\nabla^{s} u_n|^2}{2}}\inn{\nabla^{s}v_n}{\nabla^s\varphi_1}dx = \gamma(\infty)\int_{\Omega}\inn{\nabla^{s} v}{\nabla^s\varphi_1}dx=\gamma(\infty)\,\poincConst\int_{\Omega}v\varphi_1\,dx,
\]
whose proof has been relegated to the Appendix~\ref{secA1}. 
Since $v\neq 0$,  this contradicts the fact that $A> \gamma(\infty)\lambda_1^s$.   
\end{proof}

\subsection{Proof of Theorem~\ref{thm:1}}
First, we will prove the bounded condition by contradiction. Suppose, that there exists some \( u \in X \) with \( \|u\|_{H^s_0(\Omega)} = R \) such that \( \E' + \rho J = 0 \) for some \( \rho > 0 \). Then,
\[
    \int_{\Omega} \gam{\frac{|\nabla^{s} u|^2}{2}} |\nabla^{s} u|^2 dx -\int_{\Omega} f(u)u dx -\int_{\Omega} hu \,dx= \rho \int_{\Omega} \gam{\frac{|\nabla^{s} u|^2}{2}} |\nabla^{s} u|^2 dx.
\]
From here, it follows that
\[
(1 + \rho)\gmin \|u\|^2_{H^s_0(\Omega)} \leq \int_{\Omega} f(u)u \,dx +\int_{\Omega} hu \,dx.
\]
Thus, using Hölder's inequality and by \ref{item:f:4} for $R \geq t_0$, we have 

\[
R^2 < (1 + \rho)\gmin R^2 \leq Cr_h^2 + \|h\|^2_{L^2(\Omega)}R.
\]
which contradicts \eqref{condbound}. 

By Lemma \ref{lem1sol}, part (i) follows directly. On the other hand, the sequence \((u_n)_{n \in \mathbb{N}}\) is bounded in \(X\), by virtue of Proposition~\ref{cotau}. Additionally, by Lemma~\ref{lemconverg}, the sequence \((u_n)_{n \in \mathbb{N}}\) converges strongly in \(X\), up to a subsequence. Consequently, the conclusion follows from \cite[Theorems 2.9 and 2.10]{radu}.
\qed

\bigskip

\bmhead{Acknowledgements}
LC was partially supported by ANID Chile through the Postdoctoral Fondecyt Grant No.~3240062 and by Universidad de O'Higgins (UOH) through the postdoctoral position at the Institute of Engineering Sciences (ICI). AQ was partially supported by ANID Chile through the Fondecyt Grant No.~1231585.

\section*{Declarations}
Not applicable.

\begin{appendices}
\section{Convergence of the quasilinear terms}\label{secA1}
The next is a technical result, used to establish Proposition~\ref{cotau}. The proof is inspired in~\cite[Proposition~6.1]{radulescu}.\smallskip 

\begin{proposition}
Let \((u_n)_{n\in\N }\) and \((v_n)_{n\in\N}\) be the sequences defined in Proposition \ref{cotau}. Then,
\[
    \int_{\Omega} \gam{\frac{|\nabla^{s} u_n|^2}{2}}\inn{\nabla^{s}v_n}{\nabla^s w} dx \xrightarrow[n \to \infty]{} \gamma(\infty) \int_{\Omega} \inn{\nabla^{s}v}{\nabla^s w} dx, \quad \forall w\in H^s_{0}(\Omega).
\]
\end{proposition}
\begin{proof}
  Since $v_n \rightharpoonup v$ in $H_{0}^s(\Omega)$, taking  $t_n= 1/\|u_n\|_{L^2(\Omega)}$, what we need to prove is equivalent to
  \[
    \int_{\Omega} \left(\gam{\frac{|\nabla^{s} v_n|^2}{2t_n^2}} - \gamma(\infty) \right) \inn{\nabla^{s}v_n}{\nabla^s w}dx \xrightarrow[n \to \infty]{} 0.
\]
There exists $M>0$ s.t. $\sup_{n\in\N}\int_{\Omega} |\nabla^{s} v_n|^2\, dx \leq M$, as this sequence is bounded in $H_0^s(\Omega)$. 
Consider now, for $n\in\N$ and $\epsilon>0$ fixed, the following partition of $\Omega$:
\[
    A^{\epsilon}_n:=\left\{x\in \Omega :|\nabla^{s} v_n(x)|\leq\epsilon\right\} \quad \text{and}\quad B^{\epsilon}_n:= \Omega\setminus A^{\epsilon}_n.
 \]
Now, define the integrals on each component
\[
    \left(\int_{A^{\epsilon}_n}+\int_{B^{\epsilon}_n}\right)\left(\gam{\frac{|\nabla^{s} v_n|^2}{2t_n^2}} - \gamma(\infty)\right)\inn{\nabla^{s}v_n}{\nabla^sw}dx=:{\rm I}^{\epsilon}_{n}+{\rm II}^{\epsilon}_{n}.
\]
Recalling that $\gamma$ is bounded (see~\ref{item:gamma:1}), we directly get a the bound on ${\rm I}^{\epsilon}_{n}$,
\[
    \int_{A^{\epsilon}_n} \left|\gam{\frac{|\nabla^{s} v_n|^2}{2t_n^2}} - \gamma(\infty) \right||\nabla^{s}v_n| |\nabla^s w |dx \leq 2\gamma_{\max}\,\epsilon\int_{A^{\epsilon}_n} |\nabla^s w |dx \leq 2\gamma_{\max}
\sqrt{|\Omega|}\,\epsilon\,\|\nabla^s w\|_{L^2(A^{\epsilon}_n)}.
\]
Also, in view of the convergence of $\gamma(z)$ as $|z|\to\infty$~\ref{item:gamma:1}, and noting that $|\nabla^s v_n|>\epsilon$ on $B^n_m$ plus $t_n\to 0$, there exists $N_{\epsilon}\in\N$ so that
\[
    \left|\gam{\frac{|\nabla^{s} v_n|^2}{2t_n^2}} -\gamma(\infty) \right|< \epsilon\;\;\text{ on }\; B_n^m,\text{ provided }\; n\geq N_{\epsilon}.
\] 
Thus, we derive a bound on ${\rm II}^{\epsilon}_{n}$ via H\"older's inequality
\[
    \int_{B^{\epsilon}_n} \left|\gam{\frac{|\nabla^{s} v_n|^2}{2t_n^2}} - \gamma(\infty) \right|| \nabla^{s}v_n| |\nabla^s w |dx \leq \epsilon\,\| \nabla^s v_n\|_{L^2(\Omega)} \| \nabla^s w\|_{L^2(\Omega)}\leq \sqrt{M}\,\epsilon\, \| \nabla^s w\|_{L^2(\Omega)}.
\]
We have deduced that, for every $\epsilon>0$ and any $n\geq N_{\epsilon}$. 
\[
    |{\rm I}^{\epsilon}_{n}+{\rm II}^{\epsilon}_{n}|\leq\int_{\Omega} \left|\gam{\frac{|\nabla^{s} v_n|^2}{2t_n^2}} - \gamma(\infty) \right|| \nabla^{s}v_n| |\nabla^s w |dx \leq C(\Omega)\,\epsilon\,\|\nabla^s w\|_{L^2(\Omega)}.
\]
where $C(\Omega)=2\gmax\sqrt{|\Omega|}+\sqrt{M}$. The desired convergence is now established.
\end{proof}
\end{appendices}
\bigskip
\bigskip


\bibliography{sn-article.bib}% common bib file

\end{document}